\documentclass[preprint,12pt]{elsarticle}
\usepackage{amssymb}
\usepackage{amsthm}
\usepackage{tikz}
\usepackage{tikz,pgfplots}
\usetikzlibrary{decorations.markings}
\usepackage{amsmath,amssymb}

\usepackage{color}
\usepackage{graphicx}
\usepackage{caption}

\usepackage[margin=2.5cm]{geometry}% by courtesy of Mico

\journal{}

\begin{document}

	\newtheorem{theorem}{Theorem}[section]
	\newtheorem{lemma}[theorem]{Lemma}
	\newtheorem{corollary}[theorem]{Corollary}
	\newtheorem{proposition}[theorem]{Proposition}
	\newtheorem{fact}[theorem]{Fact}
	\newtheorem{observation}[theorem]{Observation}
	\newtheorem{claim}[theorem]{Claim}
	\newtheorem{definition}[theorem]{Definition}
	\newtheorem{conjecture}[theorem]{Conjecture}
	\newtheorem{problem}[theorem]{Problem}
	\newtheorem{remark}[theorem]{Remark}
	\newtheorem{example}[theorem]{Example}
	
	\newcommand{\mex}{{\rm mex}}
	\newcommand{\gex}{{\rm gex}}
	\newcommand{\gp}{{\rm gp}}
	\newcommand{\mono}{{\rm mp}}
	\newcommand{\pag}{{\rm Pag}}
	\newcommand{\mas}{{\rm Mas}}
	
	\newcommand{\TODO}[1]{\textcolor{red}{TODO: #1}}

	\begin{frontmatter}
		
		
		\title{On some extremal position problems for graphs}
		

		\author[label1]{James Tuite}
		\ead{james.tuite@open.ac.uk}
		\author[label3]{Elias John Thomas}
		\ead{eliasjohnkalarickal@gmail.com}
		\author[label2]{Ullas Chandran S. V.}
		\ead{svuc.math@gmail.com}
		
		\address[label1]{Department of Mathematics and Statistics, Open University, Walton Hall, Milton Keynes, UK}
		
		\address[label3]{Department of Mathematics, Mar Ivanios College, University of Kerala, Thiruvananthapuram-695015, Kerala, India}
		
		\address[label2]{Department of Mathematics, Mahatma Gandhi College, University of Kerala, Thiruvananthapuram-695004, Kerala, India}

		\begin{abstract}
			The general position number of a graph $G$ is the size of the largest set of vertices $S$ such that no geodesic of $G$ contains more than two elements of $S$. The monophonic position number of a graph is defined similarly, but with `induced path' in place of `geodesic'. In this paper we investigate some extremal problems for these parameters. Firstly we discuss the problem of the smallest possible order of a graph with given general and monophonic position numbers. We then determine the asymptotic order of the largest size of a graph with given general or monophonic position number, classifying the extremal graphs with monophonic position number two. Finally we establish the possible diameters of graphs with given order and monophonic position number.

		\end{abstract}
		
		
	\end{frontmatter}

	\section{Introduction}
	
	In this paper all graphs will be taken to be simple and undirected. The order of the graph $G$ will be denoted by $n$ and its size by $m$.  The \emph{clique number} $\omega (G)$ of a graph $G$ is the order of the largest clique in $G$. An \emph{independent union of cliques} in a graph $G$ is an induced subgraph $H$ of $G$ such that every component of $H$ is a clique.  The \emph{independent clique number} $\alpha ^{\omega }(G)$ is the order of a largest independent union of cliques in $G$. An \emph{extreme} vertex is a vertex with neighbourhood that induces a clique in $G$.  The distance $d(u,v)$ between two vertices $u$ and $v$ in a graph $G$ is the length of the shortest path in $G$ from $u$ to $v$ and a shortest path is called a \emph{geodesic}.  An induced path is a path without any chords; we will also call such a path \emph{monophonic}. The number of leaves (or pendant vertices) of a graph $G$ will be denoted by $\ell(G)$. For any terminology not defined here, we refer the reader to~\cite{BonMur}.  
	
	The general position problem for graphs can be traced back to one of the many puzzles of Dudeney~\cite{dudeney-1917}. This problem was introduced in the context of graph theory independently by several authors~\cite{ullas-2016,HarLouTso,ManKla2}. A set $S$ of vertices of a graph $G$ is in \emph{general position} if no geodesic of $G$ contains more than two points of $S$; in this case $S$ is a \emph{general position set}, or, for short, a \emph{gp-set}. The general position problem asks for the largest possible size of a gp-set for a given graph $G$; this number is denoted by $\gp(G)$. This problem has applications in navigation in networks~\cite{ManKla2}, social sciences~\cite{WangWangChang} and the study of social networks~\cite{Brandes}. 
	
	In~\cite{Dourado-2008,Dourado-2010} various problems in geodetic convexity in graphs were generalised to monophonic paths. In the paper~\cite{ThomasChandranTuite} the present authors introduced the `monophonic position number' of a graph, which is defined similarly to the gp-number, but with `shortest path' replaced by `induced path'; a set $S$ of vertices in a graph $G$ is in \emph{monophonic position} if there is no monophonic path in $G$ that contains more than two elements of $S$. A set satisfying this condition is called a \emph{monophonic position set} or simply an \emph{mp-set}.  The mp- and gp-numbers of trees have a particularly simple form.
	
	\begin{theorem}{\rm \cite{ullas-2016,ThomasChandranTuite}} \label{trees}
		For any tree $T$ with leaf number $\ell (T)$ we have $\mono(T) = \gp(T) = \ell(T)$.
	\end{theorem}  
	
	In this paper we consider three extremal problems for mp- and gp-sets. In~\cite{ThomasChandranTuite} it was shown that for any $a,b \in \mathbb{N}$ there exists a graph with mp-number $a$ and gp-number $b$ if and only if $2 \leq a \leq b$ or $a = b = 1$; in Section~\ref{smallest graphs} we discuss the problem of finding the smallest possible order of a graph with given mp- and gp-numbers. In Section~\ref{section:Turan} we determine the asymptotic order of the largest possible size of a graph with given order and mp-number. We also classify the graphs with mp-number two and largest size and show that the largest size of a graph with order $n$ and given gp-number is linear in $n$.  Finally in Section~\ref{section:diameter} we determine possible diameters of graphs with given order and mp-number.

	\section{The smallest graph with given mp- and gp-numbers}\label{smallest graphs}

	In a previous paper~\cite{ThomasChandranTuite} the authors proved the following realisation theorem.
	
	\begin{theorem}
		For all $a,b \in \mathbb{N}$ there exists a graph with mp-number $a$ and gp-number $b$ if and only if $2 \leq a \leq b$ or $a = b =1$. 
	\end{theorem}
	We now take this further by asking for the smallest possible order of a graph with given mp- and gp- numbers. For $2 \leq a \leq b$ we will denote the order of the smallest graph $G$ with $\mono(G) = a$ and $\gp(G) = b$ by $\mu (a,b)$. Trivially $\mu (a,a) = a$ for $a \geq 2$, so we assume that $a < b$. We introduce three families of graphs that will give strong upper bounds on $\mu (a,b)$.     
	
	\begin{lemma}\label{Pag}
		For each $r \geq 3$, there exists a graph with order $n = 3r+1$, mp-number $2$ and gp-number $2r$ and a graph with order $n = 3r$, mp-number $2$ and gp-number $2r-1$. 
	\end{lemma}
	\begin{proof}
		For $r \geq 3$ we define the \textit{pagoda graph} $\pag(r)$ as follows. The vertex set of $\pag(r)$ consists of three sets $A = \{ a_1,a_2,\dots, a_r\} , B = \{ b_1, \dots , b_r\} , C = \{ c_1, \dots , c_r\} $ of size $r$ and an additional vertex $x$.  The adjacencies of $\pag(r)$ are defined as follows. For $1 \leq i \leq r$ we set $b_i$ to be adjacent to $a_j$ and $c_j$ for $j \not = i$. Also every vertex in $C$ is adjacent to $x$.  An example $\pag(4)$ is illustrated in Figure~\ref{fig:pagoda}.  For $r \geq 3$ we will also denote the graph $\pag(r) - \{ a_r\} $ formed by deleting the vertex $a_r$ by $\pag'(r)$.

		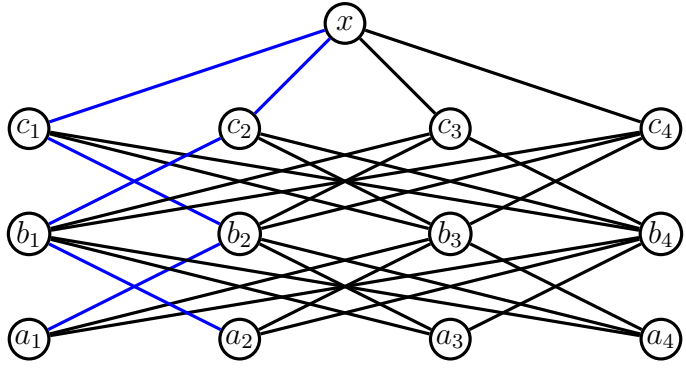
\begin{figure}
			\centering
			\begin{tikzpicture}[x=0.4mm,y=-0.4mm,inner sep=0.2mm,scale=0.7,very thick,vertex/.style={circle,draw,minimum size=15,fill=white}]

				\node at (-150,0) [vertex] (c1) {$c_1$};
				\node at (-50,0) [vertex] (c2) {$c_2$};
				\node at (50,0) [vertex] (c3) {$c_3$};
				\node at (150,0) [vertex] (c4) {$c_4$};
				
				\node at (-150,50) [vertex] (b1) {$b_1$};
				\node at (-50,50) [vertex] (b2) {$b_2$};
				\node at (50,50) [vertex] (b3) {$b_3$};
				\node at (150,50) [vertex] (b4) {$b_4$};
				
				\node at (-150,100) [vertex] (a1) {$a_1$};
				\node at (-50,100) [vertex] (a2) {$a_2$};
				\node at (50,100) [vertex] (a3) {$a_3$};
				\node at (150,100) [vertex] (a4) {$a_4$};

				\node at (0,-50) [vertex] (x) {$x$};
				
				\path
				
				(a1) edge[blue,very thick] (b2)
				(a1) edge (b3)
				(a1) edge (b4)
				(a2) edge[blue,very thick] (b1)
				(a2) edge (b3)
				(a2) edge (b4)
				(a3) edge (b1)
				(a3) edge (b2)
				(a3) edge (b4)
				(a4) edge (b1)
				(a4) edge (b2)
				(a4) edge (b3)
				
				(c1) edge[blue,very thick] (b2)
				(c1) edge (b3)
				(c1) edge (b4)
				(c2) edge[blue,very thick] (b1)
				(c2) edge (b3)
				(c2) edge (b4)
				(c3) edge (b1)
				(c3) edge (b2)
				(c3) edge (b4)
				(c4) edge (b1)
				(c4) edge (b2)
				(c4) edge (b3)
				
				(c1) edge[blue,very thick] (x)
				(c2) edge[blue,very thick] (x)
				(c3) edge (x)
				(c4) edge (x)

				;
			\end{tikzpicture}
			\caption{$\pag(4)$ with the monophonic path $P_{1,2}$ in blue}
			\label{fig:pagoda}
		\end{figure}
		The order of $\pag(r)$ is $n = 3r+1$.  We now show that for $r \geq 3$ we have $\mono(\pag(r)) = 2$ and $\gp(\pag(r)) = 2r$. Trivially any two vertices constitute an mp-set and $A \cup C$ is a gp-set of size $2r$. It therefore suffices to show that $\mono(\pag(r)) \leq 2$ and $\gp(\pag(r)) \leq 2r$. 
		
		For a contradiction, let $M$ be an mp-set in $\pag(r)$ with size $\geq 3$.  For $1 \leq i,j \leq r$ and $i \not = j$ let $P_{i,j}$ be the monophonic path $a_i,b_j,c_i,x,c_j,b_i,a_j$.  The existence of this path shows that if $x \in M$, then $M$ cannot contain two other vertices of $\pag(r)$, so we can assume that $x \not \in M$.  Suppose that $M$ contains two vertices from the same `layer' $A, B$ or $C$.  For the sake of argument say $a_1,a_2 \in M$; the other cases are similar.  The path $P_{1,2}$ shows that $b_1,b_2,c_1,c_2 \not \in M$. If another element of $A$, say $a_3$, belonged to $M$, then we would have the monophonic path $a_1,b_2,a_3,b_1,a_2$, a contradiction.  For $3 \leq i \leq r$ the path $a_1,b_i,a_2$ is trivially monophonic, so $M \cap B = \emptyset $.  It follows that there must be a point $c_i \in M$ for some $3 \leq i \leq r$.  However $a_1,b_2,c_i,b_1,a_2$ is a monophonic path, another contradiction. As $M$ cannot contain $\geq 3$ points of $\pag(r)$ we obtain the necessary inequality.         
		
		Now assume that $K$ is any gp-set in $\pag(r)$ with size $\geq 2r$. For $1 \leq i,j,k \leq r$, $j \not \in \{ i,k\} $, let $Q_{i,j,k}$ be the geodesic $a_i,b_j,c_k,x$. Suppose that $x \in K$.  If also $K \cap C \not = \emptyset $, say $c_1 \in K$, then as $c_1,x,c_i$ is a geodesic for $2 \leq i \leq r$, it follows that $K \cap \{ c_2,\dots, c_r\} = \emptyset $. Also, letting $j \not \in \{ 1,i\} $ in the path $Q_{i,j,1}$ also shows that $K \cap A = K \cap (B-\{b_1 \} ) = \emptyset $, so that we would have $|K| \leq 3 < 2r$. Hence $K \cap C = \emptyset $.  Furthermore if some $b_j$ lies in $K$, then for $1 \leq i,k \leq r$ and $j \not \in \{ i,k\} $ the geodesic $Q_{i,j,k}$ contains $x, b_j$ and $a_i$, so that we would have $K \subseteq B \cup \{ a_i,x\} $ and $|K| \leq r+2 < 2r$. Therefore $K \subseteq A \cup \{ x\} $ and $|K| \leq r+1 < 2r$.   Therefore $x$ is not contained in any gp-set of $\pag(r)$ of size $\geq 2r$.
		
		Suppose now that $K \cap B \not = \emptyset $, say $b_1 \in K$.  For $2 \leq i,k \leq r$ the existence of the geodesic $Q_{i,1,k}$ shows that $K$ cannot intersect both $A-\{a_1\} $ and $C-\{c_1\} $.  Therefore if $|K| \geq 2r+1$, $K$ must either have the form $A \cup B \cup \{ c_1\} $ or $\{ a_1\} \cup B \cup C$; however, $a_1,b_2,c_1$ is a geodesic that contains three points from both of these sets. It follows that $|K| \leq 2r$.  Furthermore, if $|K \cap B| \geq 2$, say $b_2 \in K$, then $K \cap ((A-\{ a_1,a_2\} ) \cup (C-\{ c_1,c_2\} )) = \emptyset $ and also $K$ cannot contain both $a_i$ and $c_i$ for $i = 1,2$, so that $|K|$ would be bounded above by $r+2$, which is strictly less than $2r$, whereas if $K$ contains a unique vertex of $B$, then again $|K| \leq r+2$; therefore $A \cup C$ is the unique gp-set in $\pag(r)$ with size $2r$. In a similar fashion it can be shown that the graphs $\pag'(r) = \pag(r)-\{ a_r\} $ have order $3r$, mp-number $2$ and gp-number $2r-1$ for $r \geq 3$.
	\end{proof}

	We will now define a second family of graphs.  We need a result from~\cite{ThomasChandranTuite} on the mp- and gp-numbers of the join of graphs.
	
	\begin{lemma}{\rm \cite{ThomasChandranTuite}}\label{lemma:join}
		The monophonic and general position numbers of the join $G \vee H$ of graphs $G$ and $H$ is related to the monophonic and general position numbers of $G$ and $H$ by
		\[ \mono(G \vee H) = \max\{\omega (G) + \omega (H), \mono(G), \mono(H)\}\]
		and
		\[ \gp(G \vee H) = \max\{\omega (G) + \omega (H), \alpha ^{\omega }(G), \alpha ^{\omega }(H)\}.\]
	\end{lemma}   
	
	For $r,s,t \geq 0$ let $T(r,s,t)$ be the starlike tree that has $r$ branches of length one, $s$ branches of length two and one branch of length $t$; in other words, $T(r,s,t)$ is the tree containing a vertex $x$ such that $T(r,s,t)-x$ consists of $r$ copies of $K_1$, $s$ copies of $K_2$ and one copy of $P_t$, where $P_t$ is the path of length $t-1$. We use $T(r,s,t)$ as a `base graph'. We define the graph $C(r,s,t)$ to be the join $T(r,s,t) \vee K_1$, where $K_1$ is a complete graph with vertex set $\{ y\} $. If $t = 0$ we will write $C(r,s)$ for $C(r,s,0)$.
	
	\begin{lemma}\label{chalice}
		If $t \geq 1$ and $r+s \geq 2$, then $C(r,s,t)$ has order $n = r+2s+t+2$ and we have $\mono(C(r,s,t)) = r+s+1$ and $\gp(C(r,s,t)) = r+2s+t-\lfloor \frac{t}{3}\rfloor $. Also if $r+s \geq 3$, then $C(r,s)$ has order $r+2s+2$ and $\mono(C(r,s)) = r+s$ and $\gp(C(r,s)) = r+2s$.
	\end{lemma}
	\begin{proof}
		By Theorem~\ref{trees} we have $\mono(T(r,s,t)) = r+s+1$ if $t \geq 1$ and $\mono(T(r,s,0)) = r+s$.  Therefore by Lemma~\ref{lemma:join} if $r+s+t \geq 1$ we obtain 
		\[ \mono(C(r,s,t)) = \mono(T(r,s,t) \vee K_1) = \max\{ 3,r+s+1 \}   \] 
		if $t \geq 1$ and similarly $\mono(C(r,s)) = \max\{ 3,r+s\} $.  Therefore if $t \geq 1$ and $r+s \geq 2$ then $\mono(C(r,s,t)) = r+s+1$ and if $r+s \geq 3$ then $\mono(C(r,s)) = r+s$. As $\alpha ^{\omega }(T(r,s,t)) = r+2s+t-\lfloor \frac{t}{3}\rfloor $, the result for the gp-number follows in a similar fashion.
	\end{proof}
	
	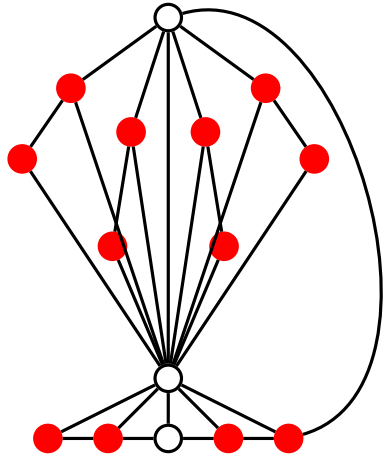
\begin{figure}
		\centering
		\begin{tikzpicture}[x=0.4mm,y=-0.4mm,inner sep=0.2mm,scale=0.4,very thick,vertex/.style={circle,draw,minimum size=10,fill=white}]
			
			\node at (0,0) [vertex] (x) {};
			
			\node at (-100,350) [vertex,red] (p5) {};
			\node at (-50,350) [vertex,red] (p4) {};
			\node at (0,350) [vertex] (p3) {};
			\node at (50,350) [vertex,red] (p2) {};
			\node at (100,350) [vertex,red] (p1) {};	
			
			\node at (80.90,58.78) [vertex,red] (u1) {};
			\node at (30.90,95.11) [vertex,red] (u2) {};
			\node at (-30.90,95.11) [vertex,red] (u3) {};
			\node at (-80.90,58.78) [vertex,red] (u4) {};

			\node at (121.35,117.56) [vertex,red] (v1) {};
			\node at (46.35,190.21) [vertex,red] (v2) {};
			\node at (-46.35,190.21) [vertex,red] (v3) {};
			\node at (-121.35,117.56) [vertex,red] (v4) {};	
			\node at (0,300) [vertex] (z) {};

			\path
			(z) edge (p1)
			(z) edge (p2)
			(z) edge (p3)
			(z) edge (p4)
			(z) edge (p5) 
			(p1) edge (p2)
			(p2) edge (p3)
			(p3) edge (p4)
			(p4) edge (p5)
			(x) edge [bend left = 90] (p1)	
			(x) edge (u1)
			(x) edge (u2)
			(x) edge (u3)
			(x) edge (u4)
			
			(u1) edge (v1)
			(u2) edge (v2)
			(u3) edge (v3)
			(u4) edge (v4)
			
			(z) edge (x)
			(z) edge (u1)
			(z) edge (u2)
			(z) edge (u3)
			(z) edge (u4)
			(z) edge (v1)
			(z) edge (v2)
			(z) edge (v3)
			(z) edge (v4)

			;
		\end{tikzpicture}
		\caption{$C(0,4,5)$ with an optimal gp-set in red}
		\label{fig:C(0,4,5)}
	\end{figure}

	We now provide a third family of extremal graphs.
	\begin{lemma}\label{Mas}
		For all $a$ and $b$ such that $3 \leq b$, $\frac{2b}{3} \leq a < b$ there exists a graph $G$ with $\mono(G) = a$, $\gp(G) = b$ and order $n = b+2$. 
	\end{lemma}
	\begin{proof}
		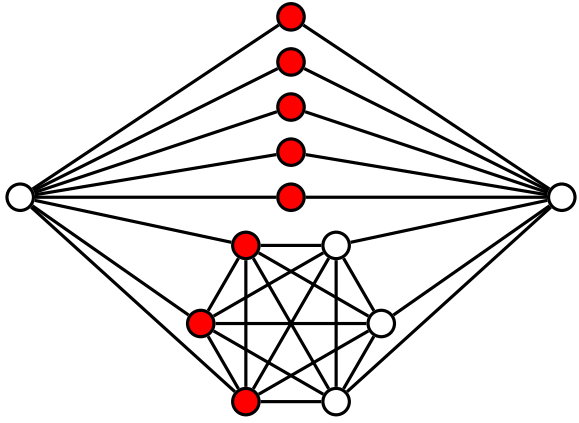
\begin{figure}
			\centering
			\begin{tikzpicture}[x=0.4mm,y=-0.4mm,inner sep=0.2mm,scale=0.6,very thick,vertex/.style={circle,draw,minimum size=10,fill=white}]
				\node at (0,50) [vertex,fill=red] (v1) {$$};
				\node at (0,25) [vertex,fill=red] (v2) {$$};
				\node at (0,0) [vertex,fill=red] (v3) {$$};
				\node at (0,-25) [vertex,fill=red] (v4) {$$};
				\node at (0,-50) [vertex,fill=red] (v5) {$$};
				\node at (-150,50) [vertex] (b) {$$};
				\node at (150,50) [vertex] (b') {$$};
				
				\node at (50,120) [vertex] (a1) {};
				\node at (25,163.3) [vertex] (a2) {};
				\node at (-25,163.3) [vertex,fill=red] (a3) {};
				\node at (-50,120) [vertex,fill=red] (a4) {};
				\node at (-25,76.7) [vertex,fill=red] (a5) {};
				\node at (25,76.7) [vertex] (a6) {};

				\path
				(a1) edge (a2)
				(a1) edge (a3)
				(a1) edge (a4)
				(a1) edge (a5)
				(a1) edge (a6)
				(a2) edge (a3)
				(a2) edge (a4)
				(a2) edge (a5)
				(a2) edge (a6)
				(a3) edge (a4)
				(a3) edge (a5)
				(a3) edge (a6)
				(a4) edge (a5)
				(a4) edge (a6)
				(a5) edge (a6)
				
				(b') edge (a6)
				(b') edge (a1)
				(b') edge (a2)
				(b) edge (a5)
				(b) edge (a4)
				(b) edge (a3)
				(b) edge (v1)
				(b) edge (v2)
				(b) edge (v3)
				(b) edge (v4)
				(b) edge (v5)	
				(b') edge (v1)
				(b') edge (v2)
				(b') edge (v3)
				(b') edge (v4)
				(b') edge (v5)

				;
			\end{tikzpicture}
			\caption{$\mas(6,5)$ with an optimal mp-set in red}
			\label{fig:b+2}
		\end{figure}
		
		Let $s \geq \lfloor \frac{r}{2} \rfloor$. We define a graph $\mas(r,s)$ as follows. Let $R$ be a set of $r$ vertices and draw the complete graph $K_r$ on $R$. Divide $R$ into two parts $R'$ and $R''$, where $|R'| = \lceil \frac{r}{2} \rceil , |R''| = \lfloor \frac{r}{2} \rfloor $. Let $S$ be an independent set with $S \cap R = \emptyset $ and $|S| = s$. Introduce two new vertices $x, y$ and connect $x$ to every vertex in $S \cup R'$ and $y$ to every vertex in $S \cup R''$. The order of $\mas(r,s)$ is $n = r+s+2$. The graph $\mas(6,5)$ is shown in Figure~\ref{fig:b+2}.
		
		As vertices in $S$ are twins and vertices in $R'$ are adjacent twins, it is easily seen that $S \cup R'$ is an mp-set in $\mas(r,s)$.  It follows that $\mono(\mas(r,s)) \geq s+\lceil \frac{r}{2} \rceil $. We show that this mp-set is optimal. Suppose that $M$ is an mp-set in $\mas(r,s)$ containing $\geq s+\lceil \frac{r}{2} \rceil +1$ vertices. If $M$ contains $x$, then $M$ can contain points in at most one of the sets $S,R',R''$, so $|M| \leq 1+\max \{ s,\lceil \frac{r}{2} \rceil \} \leq s + \lceil \frac{r}{2} \rceil $. Similarly $M$ cannot contain $y$.  $M$ does not intersect all three of $S,R'$ and $R''$, as any such three points are connected by a monophonic path, so that again $|M| \leq s+ \lceil \frac{r}{2} \rceil $. Therefore $\mono(\mas(r,s)) = s+\lceil \frac{r}{2} \rceil $.
		
		The only geodesics between two elements of $S$ or a vertex of $S$ and a vertex of $R$ have length two and pass through $x$ or $y$. Also the distance between vertices in $R$ is one. It follows that $R \cup S$ is in general position and it is clear that $\mas(r,s)$ can contain no larger gp-set. Thus $\gp(\mas(r,s)) = r+s$. For given $a,b$ in the above range setting $r = 2(b-a)$ and $s = 2a-b$ yields the required graph.  	
	\end{proof}
	
	We first prove a lower bound that shows that the graphs with order $b+2$ from Lemmas~\ref{chalice} and~\ref{Mas} are extremal.

	\begin{lemma}\label{lemma:lower bound}
		For $2 \leq a < b$ we have $\mu (a,b) \geq b+2$.
	\end{lemma}
	\begin{proof}
		Suppose that $a = \mono(G) < \gp(G) = b$ and that $G$ has order $n$. As $b > a$, $G$ is not a complete graph and so $n \geq b+1$. Let $K$ be a gp-set with order $b$. As $a \not = b$ there must exist vertices $x,y,z \in K$ such that there exists a monophonic $x,z$-path passing through $y$; let $P$ be the shortest such path in $G$. As $P$ is monophonic, there is no edge from $x$ to $z$ and thus $d(x,z) \geq 2$. Any geodesic contains at most two points of $K$, so if $d(x,z) \geq 3$, then this geodesic would contain at least two vertices outside of $K$ and $n \geq b+2$, so we can assume that $d(x,z) = 2$. As $P$ is not a geodesic, the length of $P$ is at least three. We chose $P$ to be a shortest monophonic path containing three points of $K$, so $P$ must contain a vertex outside of $K$. As $P$ is induced, the $x,z$ geodesic is internally disjoint with $P$ and we have exhibited at least two vertices of $G$ not lying in $K$.     	
	\end{proof}

	\begin{table}

		\begin{small}
			\begin{center}
				\begin{tabular}{| c||c|c| c| c| c |c |c| c| c| c|  }
					\hline
					
					$a$/$b$ & 2 & 3 & 4 & 5 & 6 & 7 & 8 & 9 & 10 & 11  \\
					\hline \hline
					2 & 2 (1) & 5 (1) & 7 (1) & 9 (4) &  10 (2) & 12 & 13 & & &\\
					\hline
					3 & - & 3 (1) & 6 (12) & 7 (7) & 8 (2) & 10 (150) & & & &\\
					\hline
					4 & - & - & 4 (1) & 7 (43) & 8 (41) & 9 (8) & 10 (2) & & & \\ 
					\hline
					5 & - & - & - & 5 (1) & 8 (104) & 9 (133) & 10 (62) & 11 & 12 & \\
					\hline
					6 & - & -&- & - & 6 (1) & 9 (219) & 10 (325) & 11 & 12 & 13\\ 
					\hline
					7 &- &- &- &- &- & 7 (1) & 10 (421) & 11 & 12 & 13\\ 
					\hline
					8 &- &- &- &- & -& -& 8 (1) & 11 & 12 & 13 \\ 
					\hline
					9 & - & - & - & -& -& -& -& 9 & 12 & 13 \\
					\hline
					10 & - &- &- &- &- &- &- &- & 10 & 13\\  
					
					\hline
				\end{tabular}
			\end{center}
		\end{small}
		\caption{$\mu (a,b)$ (number of non-isomorphic solutions in brackets)}
		\label{fig:table}
	\end{table}
	
	Table~\ref{fig:table} shows the values of $\mu (a,b)$ for some small $a$ and $b$.  This data is the result of computational work by Erskine~\cite{Erskine}. Lemmas~\ref{chalice} and~\ref{Mas} yield the following upper bounds on $\mu (a,b)$. We conjecture that equality holds for each of these bounds.
	
	\begin{theorem}\label{best upper bound}
		\text{  }
		
		\begin{itemize}
			\item For $a \geq 2$, $\mu (a,a) = a$.
			\item $\mu (2,3) = 5$ and for $b \geq 4$ we have $\mu (2,b) \leq \lceil \frac{3b}{2} \rceil +1  $, with equality for $4 \leq b \leq 8$.
			\item For $3 \leq a < b$ and $\frac{b}{2} \leq a$ we have $\mu (a,b) = b+2$.
			\item For $3 \leq a < \frac{b}{2}$ we have $\mu (a,b) \leq b-a+2 + \lceil \frac{b}{2} \rceil $. 
		\end{itemize}
		
	\end{theorem}
	\begin{proof}
		If $a = b$ the complete graph $K_a$ has the required properties. The smallest graph with $(a,b) = (2,3)$ is $C_5$ and computer search shows that the graph $H(3)$ with order $7$ from~\cite{ThomasChandranTuite} is a smallest graph with $(a,b) = (2,4)$. For $r \geq 3$ by Lemma~\ref{Pag} we have $\mono(\pag(r)) = 2$, $\gp(\pag(r)) = 2r$ and $\mono(\pag'(r)) = 2$ and $\gp(\pag'(r)) = 2r-1$, which gives the stated upper bound for $a = 2$. 
		
		Let $3 \leq a < b$ and $\frac{b}{2} \leq a$. By Lemma~\ref{chalice} the graph $C(r,s)$ has $\mono(C(r,s)) = r+s$ and $\gp(C(r,s)) = r+2s$. Solving $r+s = a, r+2s= b$ yields the values $r = 2a-b$ and $s = b-a$.  Therefore, as $C(r,s)$ has order $r+2s+2$, we have $\mu (a,b) \leq b+2$. Combined with the lower bound Lemma~\ref{lemma:lower bound} we have equality.
		
		Let $b > 2a$.  For even $b$ consider the graph $C(0,a-1,\frac{3b+4}{2}-3a)$ and for $b$ odd the graph $C(0,a-1,\frac{3b+5}{2}-3a)$. Again by Lemma~\ref{chalice} these graphs have mp-number $a$, gp-number $b$ and order as given in the statement of the theorem.
	\end{proof}
	
	The constructions in this section can be used to partially characterise the possible numbers of vertices of graphs with given mp- and gp- numbers.
	\begin{lemma}\label{intersecting mp and gp sets}
		If there exists a graph $G$ with order $n$, $\mono(G) = a$ and $\gp(G) = b$ with an optimal mp-set $M$ and an optimal gp-set $K$ such that there is an extreme vertex $v \in M \cap K$, then for any $n' \geq n$ there exists a graph $G'$ with order $n'$, $\mono(G') = a$ and $\gp(G') = b$. 
	\end{lemma}
	\begin{proof}
		Suppose that there exists a graph $G$ with $\mono(G) = a$ and $\gp(G) = b$ and order $n$ and that there exist maximal mp- and gp-sets $M$ and $K$ such that $M \cap K$ is non-empty and contains an extreme vertex $v$ of $G$. It is shown in~\cite{ThomasChandranTuite} that if $G^\prime $ is a graph obtained from $G$ by adding a pendant vertex to an extreme vertex of $G$, then $\mono(G^\prime ) = \mono(G)$ and $\gp(G^\prime ) = \gp(G)$. Hence if we add a leaf to $v$ then the resulting graph $G'$ has order $n+1$ and has the same mp- and gp-number as $G$. The new leaf in $G^\prime $ is also an extreme vertex and by~\cite{ThomasChandranTuite} there are largest mp- and gp-sets of $G'$ containing $v$; therefore we can repeatedly add leaves to construct a graph with mp-number $a$, gp-number $b$ and any order $n' \geq n$. 
	\end{proof}

	\begin{corollary}\label{maincorollary}
		\text{  }
		\begin{itemize}
			\item For $a \geq 2$, there exists a graph $G$ with $\mono(G) = \gp(G) = a$ and order $n$ if and only if $a = n = 1$ or $2 \leq a \leq n$.
			\item For $3 \leq a < b$ and $\frac{b}{2} \leq a$ there exists a graph $G$ with $\mono(G) = a$, $\gp(G) = b$ and order $n$ if and only if $n \geq b+2$.
			\item For $3 \leq a < \frac{b}{2}$ there exists a graph $G$ with $\mono(G) = a$, $\gp(G) = b$ and order $n$ if $n \geq b-a+2 + \lceil \frac{b}{2} \rceil $.
			\item There is a graph $G$ with $\mono(G) = 2$ and $\gp (G) = 3$ and order $n$ if and only if $n \geq 5$. For $b \geq 4$ there exists a graph $G$ with $\mono(G) = 2$ and $\gp(G) = b$ and order $n$ if $n \geq 2b+1$.
		\end{itemize}
	\end{corollary}
	\begin{proof}
		Adding leaves to extreme vertices of the chalice graphs and cliques yields the first three results by Lemma~\ref{intersecting mp and gp sets}. The cycles show that there is a graph with mp-number $2$ and gp-number $3$ if and only if $n \geq 5$. For larger gp-numbers Lemma~\ref{intersecting mp and gp sets} does not apply to the small constructions with mp-number $a = 2$. However by making a small adaption to the `half-wheel' graphs from~\cite{ThomasChandranTuite} we can construct graphs with $a = 2$, $b \geq 4$ and any order $n \geq 2b+1$. For $b \geq 4$ and $n \geq 2b+1$ we construct the graph $H(n,b)$ by extending one of the vertices of the `half-wheel' construction from~\cite{ThomasChandranTuite} into a path. Let $C_{n-1}$ be a cycle of length $n-1$, with vertices labelled by the elements of $\mathbb{Z}_{n-1}$ in the natural manner. The graph $H(n,b)$ is obtained by adding an extra vertex $x$ to $C_{n-1}$ and joining $x$ by edges to the vertices $0,2,4,\dots ,2(b-1)$. The proof that this graph has mp-number two and gp-number $b$ follows trivially from~\cite{ThomasChandranTuite}.
	\end{proof}

	\section{Extremal size for graphs with given position numbers}\label{section:Turan}
	
	In this section we investigate the largest possible size of graphs with given mp- or gp-number; it turns out that this problem has interesting connections with both Tur\'{a}n problems and Ramsey theory, which are fundamental areas of extremal graph theory.
	
	\begin{definition}
		For $a \geq 2$ and $n \geq a$ we define $\mex(n;a)$ (respectively $\gex(n;a)$) to be the largest possible size of a graph with order $n$ and mp-number (resp. gp-number) $a$.  
	\end{definition}
	By Theorem~\ref{trees} both of these functions are well-defined. First we determine the asymptotic behaviour of the function $\mex(n;a)$. As any clique in a graph is in monophonic position, we have the lower bound $\mono (G) \geq \omega (G)$. It follows that any graph with mp-number $a$ is $K_{a+1}$-free. The largest possible size of a $K_{a+1}$-free graph was determined by Tur\'{a}n.
	
	\begin{theorem}{\rm \cite{Tur}}
		The number of edges of a $K_{a+1}$-free graph $H$ is at most \[{n-r \choose 2} + (a-1) {r+1 \choose 2},\]
		where $r = \lfloor \frac{n}{a} \rfloor $.  Equality holds if and only if $H$ is isomorphic to the Tur\'{a}n graph $T_{n,a}$, which is the complete $a$-partite graph with every partite set of size $\lfloor \frac{n}{a} \rfloor $ or $\lceil \frac{n}{a} \rceil $.
	\end{theorem} 
	
	The size of the Tur\'{a}n graph $T_{n,a}$ will be denoted by $t_{n,a}$. We will show that $\mex(n;a)$ has the same asymptotic behaviour as the largest size of a $K_{a+1}$-free graph; our strategy is to construct a graph with small mp-number by discarding a linear number of cliques from the Tur\'{a}n graph $T_{n,a}$. We need a lemma on the mp-numbers of complete multipartite graphs that extends the result on complete bipartite graphs from~\cite{ThomasChandranTuite}.
	
	\begin{lemma}\label{multipartite}
		For integers $r_1\geq r_2 \geq \dots \geq r_t$ the mp-number and gp-number of the complete multipartite graph $K_{r_1,r_2,\dots ,r_t}$ are given by
		\[ \gp(K_{r_1,r_2,\dots ,r_t}) = \mono(K_{r_1,r_2,\dots ,r_t}) = \max \{ r_1,t\} . \]  
	\end{lemma}
	\begin{proof}
		Let the partite sets of $K_{r_1,r_2,\dots ,r_t}$ be $W_1,W_2, \dots ,W_t$ and let $M$ be a maximum mp-set of $K_{r_1,r_2,\dots ,r_t}$. Suppose that $M$ contains two vertices $u_1,u_2$ in the same partite set $W$. Then $M$ cannot contain any vertex $v$ in any other partite set, for $u_1,v,u_2$ is a monophonic path. Hence in this case $|M| \leq |W| \leq r_1$. If $M$ contains at most one vertex from every partite set then $|M| \leq t$. For the converse, observe that $K_{r_1,r_2,\dots ,r_t}$ contains a clique of size $t$, so that $|M| \geq t$. Each partite set is also an mp-set, so that $|M| \geq r_1$. The proof for gp-sets is identical.    
	\end{proof}
	
	\begin{corollary}\label{Turan upper bound}
		For $a \leq n \leq a^2$ we have $\mex(n;a) = t_{n,a}$ and $\gex(n;a) =  t_{n,a}$, whilst $\mex(n;a) < t_{n,a}$ and $\gex(n;a) < t_{n,a}$ for $n \geq a^2+1$. 
	\end{corollary} 
	\begin{proof}
		If a graph contains a clique of size $\geq a+1$, then it will have mp- and gp-number $\geq a+1$. Therefore any graph with mp- or gp-number $a$ is $K_{a+1}$-free and it follows from Tur\'{a}n's Theorem that $\gex(n;a) \leq t_{n,a}$ and $\mex(n;a) \leq t_{n,a}$. As the Tur\'{a}n graph is the unique extremal graph, we have equality if and only if the Tur\'{a}n graph has mp-number $a$. By Lemma~\ref{multipartite}, $T_{n,a}$ does have mp-number $a$ for $a \leq n \leq a^2$, but for $n \geq a^2+1$ $T_{n,a}$ has mp- and gp-number $\lceil \frac{n}{a} \rceil \geq a+1$. 	
	\end{proof} 
	
	We now present the main theorem of this section.
	
	\begin{theorem}\label{turantheorem}
		For $a \geq 2$ and $n \geq a^2+1$ we have
		\[ t_{n,a}-\left \lfloor \frac{n}{a} \right \rfloor {a \choose 2} \leq \mex(n;a) \leq t_{n,a}-1. \]
		Thus $\mex(n;a) \sim (1-\frac{1}{a})\frac{n^2}{2}$.
	\end{theorem}   
	\begin{proof}
		Take the Tur\'{a}n graph $T_{n,a}$ and label the partite sets $T_1, T_2, \dots , T_a$, where $|T_i| = \lceil \frac{n}{a} \rceil $ for $1 \leq i \leq s$ and $|T_i| = \lfloor \frac{n}{a} \rfloor $ for $s+1 \leq i \leq a$, where $s = n-\lfloor \frac{n}{a}\rfloor a$. For $1 \leq i \leq a$ we denote the vertices of $T_i$ by $u_{ij}$, where $1 \leq j \leq \lceil \frac{n}{a} \rceil $ if $i \leq s$ and $1 \leq j \leq \lfloor \frac{n}{a} \rfloor $ if $s+1 \leq i \leq a$.
		
		For each $j$ in the range $1 \leq j \leq \lfloor \frac{n}{a} \rfloor $ delete the edges of the clique of order $a$ on the vertices $u_{ij}$, $1 \leq i \leq a$, from $T_{n,a}$. This yields a new graph $T^*_{n,a}$ with size $t_{n,a} - \lfloor \frac{n}{a} \rfloor {a \choose 2}$.    
		
		Considering the vertices $u_{ii}$ for $1 \leq i \leq a$, we see that $T^*_{n,a}$ contains a clique of size $a$, so certainly $\mono(T^*_{n,a}) \geq a$.  For the converse, let $M$ be a largest mp-set of $T^*_{n,a}$.  Suppose that $M$ contains two vertices $u_{ij}$ and $u_{ik}$ from the same partite set $T_i$. Then $M$ cannot contain a vertex $u_{i'j'}$ from a different partite set $T_{i'}$ where $j' \not \in \{ j,k\} $, as $u_{ij},u_{i'j'},u_{ik}$ is a monophonic path. We have $\mex(5;2) = 5$, so the lower bound holds for $n = 5$ and $a = 2$; otherwise for $n \geq a^2+1$ we have $\lfloor \frac{n}{a} \rfloor \geq 3$.  Hence let $1 \leq l \leq \lfloor \frac{n}{a} \rfloor$ and $l \not \in \{ j,k\} $. Then for any $i' \in \{ 1,2,\dots , a\} - \{ i\} $ the path $P = u_{ij},u_{i'k},u_{il},u_{i'j},u_{ik}$ is monophonic, so $M \subseteq V(T_i)$.  However the path $P$ shows that no three points of the partite set $T_i$ can all lie in $M$, so that $|M| \leq 2$. Therefore we can assume that any optimal mp-set has at most one point in each partite set, so that $\mono(T^*_{n,a}) \leq a$, completing the proof.	
	\end{proof}  
	In general the construction of Theorem~\ref{turantheorem} is not optimal; for example $\mex (10,3) = 31$, whereas Theorem~\ref{turantheorem} gives a lower bound of 24. We now give an exact formula for $\mex(n;2)$.
	
	\begin{theorem}
		For $n \geq 6$ we have $\mex(n;2) = \left \lceil \frac{(n-1)^2}{4} \right \rceil $. For odd $n$ the unique extremal graph is given by $T_{n,2}^*$, whilst for even $n$ the unique extremal graph is $T_{n,2}^*$ with one edge added between the partite sets. 
	\end{theorem} 
	\begin{proof}
		It is easily verified that for even $n$ the graph formed from the construction in Theorem~\ref{turantheorem} by adding an extra edge between the partite sets (i.e. the graph $K_{r,r}$ minus a matching of size $r-1$) has mp-number two, so that $\mex(2r;2) \geq r^2-r+1$ and $\mex(2r+1;2) \geq r^2$. Computer search confirms that these are the unique extremal graphs for $5 \leq n \leq 12$~\cite{Erskine}. 
		
		Now we prove the upper bound. The stability result of~\cite{Bro} states that for $n \geq 2a+1$ any $K_{(a+1)}$-free graph with $\geq t_{n,a}-\lfloor \frac{n}{a} \rfloor +2$ edges must be $a$-partite. In particular, for $n \geq 5$ any triangle-free graph with order $2r$ and size $\geq r^2-r+2$ or order $2r+1$ and size $\geq r^2+2$ must be bipartite. Let $G$ be any bipartite graph with partite sets $X$ and $Y$, where $|X| \geq |Y|$, and monophonic position number two. Suppose that there are $\geq 2$ vertices $x_1,x_2$ in $X$ that are connected to every vertex of $Y$. Then for any $x_3 \in X- \{ x_1,x_2\} $ the set $\{ x_1,x_2,x_3\} $ is in monophonic position. It follows that at most one vertex of $X$ is adjacent to every vertex of $Y$, so that the size of $G$ is at most $\lfloor \frac{n^2}{4}\rfloor - \lceil \frac{n}{2} \rceil + 1$. Therefore $\mex(2r;2) = r^2-r+1$ and $r^2 \leq \mex(2r+1;2) \leq r^2+1$. 
		
		Suppose that $G$ is a graph with order $n = 2r+1$, size $r^2+1$ and mp-number two. By the previous argument $G$ is not bipartite. For any graph $H$ with $k$ vertices $v_1,v_2,\dots ,v_k$ and positive integers $n_1,n_2,\dots ,n_k$ we let $H[n_1,n_2,\dots ,n_k]$ denote the blow-up graph obtained from $H$ by replacing the vertex $v_i$ by a set $V_i$ of $n_i$ vertices and joining every vertex in $V_i$ to every vertex in $V_j$ if and only if $v_i \sim v_j$ in $H$. It is shown in~\cite{AminFauGouSid} that the only non-bipartite triangle-free graphs with order $2r+1$ and size $r^2+1$ are given by the blow-up $C_5[r-1,k,1,1,r-k]$ of the 5-cycle, where $1 \leq k \leq r-1$. The union of the partite sets with size $k$ and $r-k$ in this expanded graph is in monophonic position, so for $r \geq 3$ no such graph exists. Hence $\mex(2r+1;2) = r^2$ for $r \geq 3$.   
		
		We now classify the extremal graphs. Again by~\cite{AminFauGouSid} the only non-bipartite triangle-free graphs with order $2r$ and size $r^2-r+1$ are the blow-ups $C_5[r-2,k,1,1,r-k]$ ($1 \leq k \leq r-1$) and $C_5[r-1,k,1,1,r-k-1]$ ($1 \leq k \leq r-2$), which both have mp-number $\geq 3$ for $r \geq 3$, so any extremal graph must be bipartite. Applying our counting argument for bipartite graphs shows that both parts of the bipartition must have size $r$ and there is a matching of size $r-1$ missing between $X$ and $Y$; hence the graph $T_{2r,2}^*$ with an edge added between the partite sets is the unique extremal graph. 
		
		Now let $G$ be a graph with order $n = 2r+1$, mp-number two and size $r^2$. For $n = 7, 9$ or $11$ the result can be shown by computer search~\cite{Erskine} or a slightly more involved argument, so we assume that $r \geq 6$. First we show that $G$ must be bipartite. Let $\delta $ be the minimum degree of $G$ and $x$ be a vertex of $G$ with degree $\delta $. There must be a vertex of degree $\leq r-1$ in $G$, for otherwise the size of $G$ would be at least $\frac{1}{2}(2r+1)r > r^2$. Thus $\delta \leq r-1$. Then $G' = G-x$ is a triangle-free graph with order $2r$ and size $r^2-\delta \geq r^2-r+1$. It follows that either $G'$ is bipartite, or $\delta = r-1$ and $G'$ is isomorphic to one of the blow-ups $C_5[r-2,k,1,1,r-k]$ ($1 \leq k \leq r-1$) or $C_5[r-1,k,1,1,r-k-1]$ ($1 \leq k \leq r-2$). 
		
		Suppose that $G'$ is bipartite with bipartition $(X',Y')$, where we make no assumption about the relative sizes of $X'$ and $Y'$. Write $a = |N_G(x) \cap X'|$ and $b = |N_G(x) \cap Y'|$, where $a+b= \delta $. If either $a = 0$ or $b = 0$, then $G$ is bipartite. $G'$ has size at most $|X'||Y'| \leq r^2$ and, as $G$ is triangle-free, there are at least $ab$ edges missing between $N_G(x) \cap X'$ and $N_G(x) \cap Y'$, so the size $r^2$ of $G$ is bounded above by $r^2+a+b-ab$. It follows that $ab \leq a+b$, so that either $a = 1$ or $b = 1$, or else $a = 2,b = 2$. If $a = 2, b = 2$, then we have equality in the preceding bound, so that $G'$ must be isomorphic to the complete bipartite graph $K_{r,r}$ with the four edges between $N_G(x) \cap X'$ and $N_G(x) \cap Y'$ deleted, in which case the neighbourhood $N(x)$ of $x$ is a set of four vertices in monophonic position, a contradiction. Hence without loss of generality we can assume that $a = 1$; let $x'$ be the neighbour of $x$ in $X'$.  As there are $b$ edges missing between $N_G(x) \cap X'$ and $N_G(x) \cap Y'$, the size of $G$ is bounded above by $|X'||Y'|+1$; thus if either set $X'$ or $Y'$ has size $\leq r-2$, then the size of $G$ would be too small, so we can take $|X'| \geq r-1$. As before, at most one vertex in $X'-\{ x'\} $ can be adjacent to every vertex of $Y'$, so there are at least a further $r-3$ edges missing between $X'-\{ x'\} $ and $Y'$, implying that $r^2 \leq r^2-r+4$, or $r \geq 4$. This shows that for $r \geq 5$ if the graph $G'$ is bipartite, then $G$ is bipartite.     
		
		Now consider the case $\delta = r-1$ and $G' = C_5[r-2,k,1,1,r-k]$.  There must be at least one edge in $G$ from $x$ to the partite set of size $r-2$ in $G'$, for otherwise this set of $r-2$ vertices would be in monophonic position in $G$. As $G$ is triangle-free, this means that $x$ has no edges to the sets of size $k$ or $r-k$. Both of these sets will thus constitute mp-sets in $G$, so that $k\leq 2$ and $r-k \leq 2$, so that $r \leq 4$. Similar reasoning shows that $G'$ cannot be isomorphic to the blow-up $C_5[r-1,k,1,1,r-k-1]$ for $r \geq 6$. Therefore for $r \geq 6$ we can assume that $G$ is bipartite with bipartition $(X,Y)$, where $|X| > |Y|$. 
		
		Set $t = |X|$. Our bipartite counting argument shows that at most one vertex of $X$ is adjacent to every vertex of $Y$ and so there are at most $f(t) = t(2r+1-t)-(t-1) = -t^2+2rt+1$ edges in $G$. We have $f(r+1) = r^2$ and $f$ is decreasing for $t \geq r$, so $|X| = r+1$ and, as we have equality in the bound, exactly $r$ vertices in $X$ have one edge missing to $Y$, with the remaining vertex of $X$ adjacent to every vertex of $X$. Suppose that two vertices of $X$ are non-adjacent to the same vertex of $Y$, say $x_1 \not \sim y_1$ and $x_2 \not \sim y_1$; then for any $x_3 \in X-\{ x_1,x_2\} $ the set $\{ x_1,x_2,x_3\} $ would be in monophonic position. It follows that the missing edges between $X$ and $Y$ constitute a matching of size $r$ in the complement of $G$ and we recover our construction $T_{2r+1,2}^*$. \end{proof}
	We turn now to the largest size of graphs with given general position number. As the only graphs with gp-number two are $C_4$ and paths of length $\geq 2$ we trivially have $\gex(4;2) = 4$ and $\gex(n;2) = n-1$ for $n \geq 5$. A graph with order $n = 10$, general position number $3$ and largest size is shown in Figure~\ref{gex graph}. In contrast to the quadratic size of extremal graphs with given mp-number, the function $\gex(n;a)$ is $O(n)$. This can be shown by a simple upper bound on the maximum degree of such a graph that comes from Ramsey theory. The Ramsey number $R(s,t)$ is the smallest value of $n$ such that any graph with order $n$ contains either a clique of size $s$ or an independent set of size $t$; taking the converse of the extremal graphs we trivially have the symmetry $R(s,t) = R(t,s)$.
	
	\begin{theorem}\label{Ramsey upper bound}
		For $a \geq 3$ the function $\gex(n;a)$ is bounded above in terms of the Ramsey number $R(a,a+1)$ by $\gex(n;a) \leq \frac{R(a,a+1)-1}{2}n$.
	\end{theorem} 
	\begin{proof}
		Let $G$ be a graph with order $n$, gp-number $a$, size $\gex(n;a)$ and maximum degree $\Delta $. Suppose that $G$ has a vertex $x$ with degree $d(x) \geq R(a,a+1)$ and let $X$ be the subgraph induced by $N(x)$. Then $X$ contains either a clique of order $a$, which together with $x$ would give a clique of size $a+1$, or else $X$ has an independent set of size $a+1$; either of these sets constitutes a general position set with more than $a$ vertices. Thus $\Delta \leq R(a,a+1)-1$ and $\gex(n;a) \leq \frac{R(a,a+1)-1}{2}n$.
	\end{proof}
	It seems unlikely that the bound in Theorem~\ref{Ramsey upper bound} is tight; improving this bound is an interesting problem.

	\begin{figure}
		\centering
		\begin{tikzpicture}[x=0.4mm,y=-0.4mm,inner sep=0.2mm,scale=0.6,very thick,vertex/.style={circle,draw,minimum size=10,fill=lightgray}]
			\node at (70,0) [vertex] (x1) {$$};
			\node at (100,-100) [vertex] (x2) {$$};
			\node at (-100,-100) [vertex] (x3) {$$};
			\node at (-70,0) [vertex] (x4) {$$};
			\node at (-100,100) [vertex] (x5) {$$};
			\node at (100,100) [vertex] (x6) {$$};
			
			\node at (0,-140) [vertex] (y) {$$};
			\node at (0,0) [vertex] (z) {$$};
			
			\node at (0,-170) [vertex] (y') {$$};
			\node at (0,70) [vertex] (z') {$$};
			
			\path
			(x1) edge (x2)
			(x2) edge (x3)
			(x3) edge (x4)
			(x4) edge (x5)
			(x5) edge (x6)
			(x6) edge (x1)
			(x1) edge (x3)
			(x2) edge (x4)
			(x3) edge (x5)
			(x4) edge (x6)
			(x5) edge (x1)
			(x6) edge (x2)

			(z) edge (x2)
			(z) edge (x3)
			(z) edge (x5)
			(z) edge (x6)
			(x1) edge[bend right] (y)
			(x4) edge[bend left] (y)
			(y) edge (y')
			(z) edge (z')
			
			;
		\end{tikzpicture}
		\caption{A graph with order 10, gp-number 3 and largest size}
		\label{gex graph}
	\end{figure}
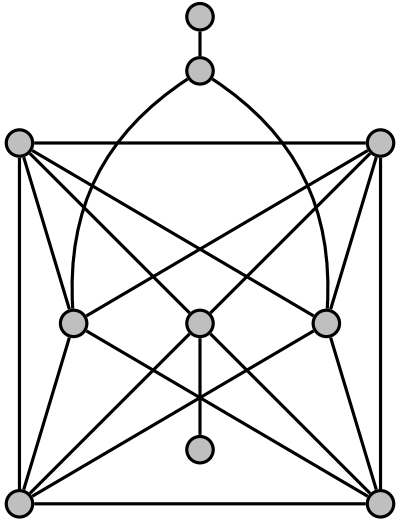

	We now briefly consider the problem of the smallest size of graph with given position numbers. Finding the smallest size of a graph with given mp- or gp-number is trivial by Theorem~\ref{trees}; however, if we specify both the mp- and gp-number this problem becomes more difficult. We will denote the smallest possible size of a graph $G$ with order $n$, $\mono (G) = a$ and $\gp(G) = b$ by $ex^-(n;a;b)$; by Corollary~\ref{maincorollary} this number exists for sufficiently large $n$. Trivially $ex^-(a;a;a) = {a \choose 2}$ and $ex^-(n;a;a) = n-1$ for $n \geq a+1$.  Also using cycles we see that $ex^-(n;2;3) = n$ for $n\geq 5$. For other values of $a$ and $b$ we conjecture that the following construction is extremal. 
	
	\begin{theorem}\label{smallest size}
		If $b$ and $a$ have the same parity, then for $n \geq \frac{5b-3a}{2}+4$ we have $ex^-(n;a;b) \leq n+\frac{b-a}{2}+1$.	If $a$ and $b$ have opposite parities, then for $n \geq \frac{5b-3a+11}{2}$ we have $ex^-(n;a;b) \leq n+\frac{b-a+3}{2}$. 
	\end{theorem}
	\begin{proof}
		For $r \geq 2$ and $t \geq 0$ we define a graph $S(r,t)$ as follows. Take a cycle $C_{5r+1}$ of length $5r+1$ and identify its vertex set with $\mathbb{Z}_{5r+1}$ in the natural way. Join the vertex $0$ to the vertices $3+5s$ for all $s \in \mathbb{N}$ in the range $0 \leq s \leq r-1$. Finally append a set $W = \{ w_1,\dots , w_t\} $ of $t$ pendant vertices to the vertex $0$. An example is shown in Figure ~\ref{fig:smallsize}. We claim that $S(r,t)$ has $\mono(S(r,t)) = t+2$ and $\gp(S(r,t)) = 2r+t$.

		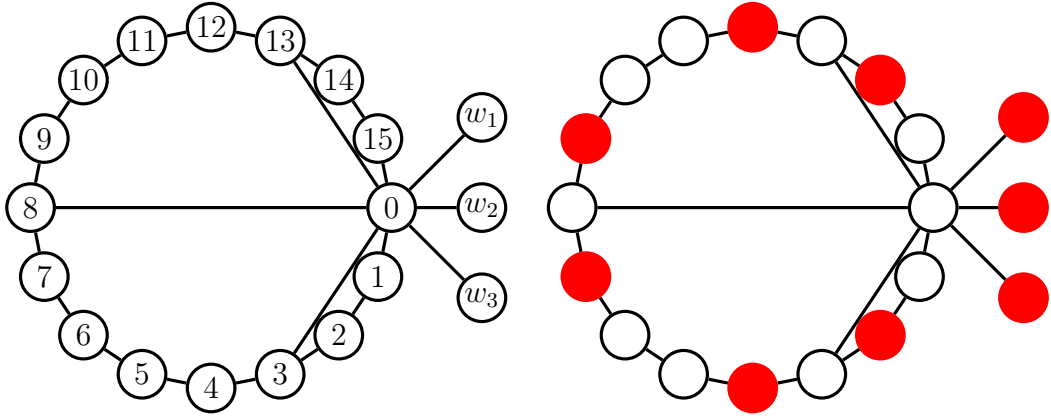
\begin{figure}
			\centering
			\begin{tikzpicture}[x=0.4mm,y=-0.4mm,inner sep=0.2mm,scale=0.6,very thick,vertex/.style={circle,draw,minimum size=18,fill=white}]
				\node at (150,-50) [vertex] (w1) {$w_1$};
				\node at (150,0) [vertex] (w2) {$w_2$};
				\node at (150,50) [vertex] (w3) {$w_3$};
				\node at (100,0) [vertex] (0) {0};
				\node at (92.39,38.27) [vertex] (1) {1};
				\node at (70.71,70.71) [vertex] (2) {2};
				\node at (38.27,92.39) [vertex] (3) {3};
				\node at (0,100) [vertex] (4) {4};
				\node at (-38.27,92.39) [vertex] (5) {5};
				\node at (-70.71,70.71) [vertex] (6) {6};
				\node at (-92.39,38.27) [vertex] (7) {7};
				\node at (-100,0) [vertex] (8) {8};
				\node at (-92.39,-38.27) [vertex] (9) {9};
				\node at (-70.71,-70.71) [vertex] (10) {10};
				\node at (-38.27,-92.39) [vertex] (11) {11};
				\node at (0,-100) [vertex] (12) {12};
				\node at (38.27,-92.39) [vertex] (13) {13};
				\node at (70.71,-70.71) [vertex] (14) {14};
				\node at (92.39,-38.27) [vertex] (15) {15};
				
				\node at (450,-50) [vertex,red] (v1) {};
				\node at (450,0) [vertex,red] (v2) {};
				\node at (450,50) [vertex,red] (v3) {};
				\node at (400,0) [vertex] (u0) {};
				\node at (392.39,38.27) [vertex] (u1) {};
				\node at (370.71,70.71) [vertex,red] (u2) {};
				\node at (338.27,92.39) [vertex] (u3) {};
				\node at (300,100) [vertex,red] (u4) {};
				\node at (261.73,92.39) [vertex] (u5) {};
				\node at (229.29,70.71) [vertex] (u6) {};
				\node at (207.61,38.27) [vertex,red] (u7) {};
				\node at (200,0) [vertex] (u8) {};
				\node at (207.61,-38.27) [vertex,red] (u9) {};
				\node at (229.29,-70.71) [vertex] (u10) {};
				\node at (261.73,-92.39) [vertex] (u11) {};
				\node at (300,-100) [vertex,red] (u12) {};
				\node at (338.27,-92.39) [vertex] (u13) {};
				\node at (370.71,-70.71) [vertex,red] (u14) {};
				\node at (392.39,-38.27) [vertex] (u15) {};
				
				\path
				(0) edge (1)
				(1) edge (2)
				(2) edge (3)
				(3) edge (4)
				(4) edge (5)
				(5) edge (6)
				(6) edge (7)
				(7) edge (8)
				(8) edge (9)
				(9) edge (10)
				(10) edge (11)
				(11) edge (12)
				(12) edge (13)
				(13) edge (14)
				(14) edge (15)
				(15) edge (0)
				
				(0) edge (3)
				(0) edge (8)
				(0) edge (13)
				
				(w1) edge (0)
				(w2) edge (0)
				(w3) edge (0)
				
				(u0) edge (u1)
				(u1) edge (u2)
				(u2) edge (u3)
				(u3) edge (u4)
				(u4) edge (u5)
				(u5) edge (u6)
				(u6) edge (u7)
				(u7) edge (u8)
				(u8) edge (u9)
				(u9) edge (u10)
				(u10) edge (u11)
				(u11) edge (u12)
				(u12) edge (u13)
				(u13) edge (u14)
				(u14) edge (u15)
				(u15) edge (u0)
				
				(u0) edge (u3)
				(u0) edge (u8)
				(u0) edge (u13)
				
				(v1) edge (u0)
				(v2) edge (u0)
				(v3) edge (u0)

				;
			\end{tikzpicture}
			\caption{$S(3,3)$ (left) with an optimal gp-set in red (right)}
			\label{fig:smallsize}
		\end{figure}

		The set $W \cup \{ 1,-1\} $ is obviously in monophonic position, so $\mono(S(r,t)) \geq t+2$. Let $M$ be any optimal mp-set of $S(r,t)$. By a result of~\cite{ThomasChandranTuite} on triangle-free graphs any set in monophonic position in $S(r,t)$ is an independent set. The path $1,2,3,\dots ,5r-1,5r$ in $C_{5r+1}$ is monophonic in $S(r,t)$ and hence contains at most two points of $M$, so if the mp-number of $S(r,t)$ is any greater than $t+2$, then $M$ contains three vertices of $C_{5r+1}$, one of which is $0$, so that $M \cap W = \emptyset $ and $t = 0$. A simple argument shows that the mp-number of $S(r,0)$ is two. Thus $\mono(S(r,t)) = t+2$.
		
		Consider now the set $\{2+5s,4+5s: 0 \leq s \leq r-1\} \cup W$. The vertices of this set are at distance at most four from each other and it is easily verified that none of the geodesics between them pass through other vertices of the set. Thus $\gp(S(r,t)) \geq 2r+t$. Let $K$ be a gp-set of $S(r,t)$ that contains $\geq 2r+t+1$ vertices. For $0 \leq s \leq r-1$ the set $S[s] = \{ 1+5s,2+5s,3+5s,4+5s,5+5s\} $ on $C_{5r+1}$ contains at most two vertices of $K$. It follows that $K$ must contain the vertex $0$, two vertices in each of the aforementioned sets and every vertex of $W$. As the vertices of $W$ have shortest paths to the vertices of $K$ in $S[0]$, we must have $t = 0$. For $r \geq 3$, if $0 \leq s < s' \leq r-1$ and $s+2 \leq s'$, then vertices in $S[s]$ have shortest paths to the vertices in $S[s']$ passing through $0$, so $0 \not \in K$ and $\gp(S(r,t)) = 2r+t$. Also $\gp(S(2,0)) = 4$.  If $a$ and $b$ have the same parity and $b > a$ the graph $S(\frac{b-a}{2}+1,a-2)$ therefore has the required parameters. By Corollary~\ref{intersecting mp and gp sets} if $a \geq 3$ we can add a path to a vertex of $W$ to give a graph with any larger order $n' \geq n$ and the same mp- and gp-numbers. If $a = 2$ then lengthening one of the sections of length five on $C_{5r+1}$ accomplishes the same aim. If $a$ and $b$ have opposite parities, then shortening one of the sections of length five on $C_{5r+1}$ in the above constructions by one vertex yields the required graph.             \end{proof}

	\section{The diameters of graphs with given order and mp-number}\label{section:diameter}
	
	In~\cite{Ost} Ostrand proved the well-known realisation result that for any two positive integers $a,b$ with $a \leq b \leq 2a$ there exists a connected graph with radius $a$ and diameter $b$. Similar realisation results for the diameters of graphs with given position or hull numbers are given in~\cite{CharHarZha,ChaZha}. This raises the following question: what are the possible diameters of a graph with given order and monophonic position number? We now solve this problem, beginning with mp-numbers $a \geq 3$.   
	\begin{theorem} 
		For any integers $a$ and $n$ with $3 \leq a \leq n-1$, there exists a connected graph $G$ with order $n$, mp-number $a$ and diameter $D$ if and only if $2 \leq D \leq n-a+1$. 
	\end{theorem}
	\begin{proof}
		For $a \geq 3$ the only connected graph with monophonic position number $a = n$ is the complete graph $K_n$ with diameter one. It was shown in~\cite{ThomasChandranTuite} that the mp-number of a graph with order $n$ and longest monophonic path with length $L$ is bounded above by $n-L+1$; rearranging, it follows that $L \leq n-a+1$. As any geodesic in $G$ is induced, it follows that $n-a+1$ is the largest possible diameter of a graph $G$ with order $n$ and $\mono(G) = a$. It remains only to show existence of the required graphs for the remaining values of the parameters. For $a = n-1$ and $D = 2$ this follows easily by considering the star graph $K_{1,n-1}$, so we can assume that $a \leq n-2$. For $n \geq 2$, Theorem~\ref{trees} shows that any caterpillar graph formed by adding $a-2$ leaves to the internal vertices of a path of length $n-a+1$ has order $n$, mp-number $a$ and diameter $D = n-a+1$.   
		
		For $a \geq 3$ we can construct a graph $F(n,a,n-a)$ with order $n$, mp-number $a$ and diameter $D = n-a$ as follows. Take a path $P$ of length $n-a$; let $V(P) = \{ u_0, u_1, \dots , u_{n-a}\} $, where $u_i \sim u_{i+1}$ for $0 \leq i \leq n-a-1$. Introduce a set $Q$ of $a-1$ new vertices $v_1,v_2, \dots, v_{a-1}$ and join each of them to $u_0$ and $u_1$. An example is shown in Figure~\ref{fig:n-k}. 
		
		We claim that the set $Q \cup \{ u_{n-a} \} $ is an mp-set. Any monophonic path $P'$ from $u_{n-a}$ to a vertex of $Q$ must pass through $u_1$, which is adjacent to every member of $Q$, and so $P'$ cannot terminate in another vertex of $Q$.  Similarly any monophonic path between two vertices of $Q$ passes through either $u_0$ or $u_1$ and hence cannot include any other vertex of $Q$ or $u_{n-a}$. Thus $\mono(F(n,a,n-a)) \geq a$.  
		
		Conversely, if $M$ is any mp-set in $F(n,a,n-a)$ with size $\geq a+1$, then $M$ contains at most two vertices from $P$ and so has size exactly $a+1$ and consists of $Q$ together with two vertices $x, y$ of $P$.  If $x = u_i, y = u_j$, where $1 \leq i < j \leq n-a$, then there is an induced path from $y$ to $Q$ through $x$, which is impossible. Hence we can take $x = u_0$. However $v_1 \sim u_0 \sim v_2$ would be a monophonic path in $M$, so we conclude that $\mono(F(n,a,n-a)) = a$.

		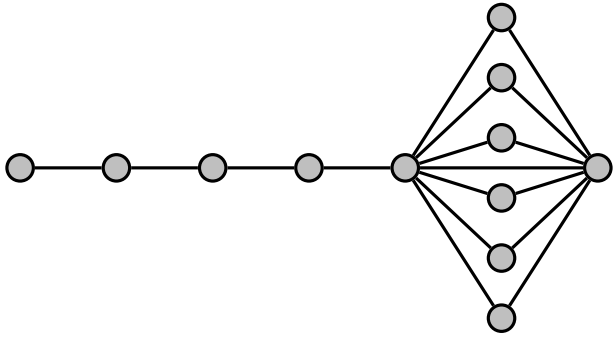
\begin{figure}
			\centering
			\begin{tikzpicture}[x=0.4mm,y=-0.4mm,inner sep=0.2mm,scale=0.4,very thick,vertex/.style={circle,draw,minimum size=10,fill=lightgray}]

				\node at (0,0) [vertex] (x1) {};
				\node at (-160,0) [vertex] (x2) {};
				\node at (-240,0) [vertex] (x3) {};
				\node at (-320,0) [vertex] (x4) {};
				\node at (-400,0) [vertex] (x5) {};
				\node at (-480,0) [vertex] (x6) {};
				\node at (-80,-125) [vertex] (u1) {};  
				\node at (-80,-75) [vertex] (u2) {}; 
				\node at (-80,-25) [vertex] (u3) {};  
				\node at (-80,25) [vertex] (u4) {};  
				\node at (-80,75) [vertex] (u5) {}; 
				\node at (-80,125) [vertex] (u6) {};      
				\path

				(x1) edge (x2)
				(x2) edge (x3)
				(x3) edge (x4)
				(x4) edge (x5)
				(x5) edge (x6)
				
				(x1) edge (u1)
				(x1) edge (u2)
				(x1) edge (u3)
				(x1) edge (u4)
				(x1) edge (u5)
				(x1) edge (u6)
				
				(x2) edge (u1)
				(x2) edge (u2)
				(x2) edge (u3)
				(x2) edge (u4)
				(x2) edge (u5)
				(x2) edge (u6)	
				
				;
			\end{tikzpicture}
			\caption{A graph $F(12,7,5)$ with order $n = 12$, mp-number $7$ and diameter $D = 5$}
			\label{fig:n-k}
		\end{figure}

		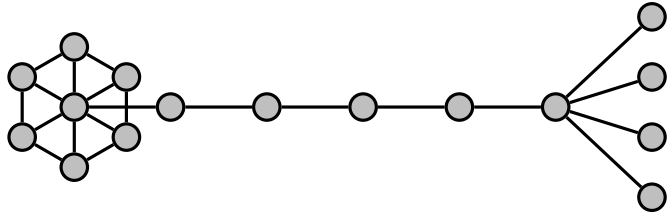
\begin{figure}
			\centering
			\begin{tikzpicture}[x=0.4mm,y=-0.4mm,inner sep=0.2mm,scale=0.4,very thick,vertex/.style={circle,draw,minimum size=10,fill=lightgray}]
				\node at (43.3,25) [vertex] (v1) {};
				\node at (0,50) [vertex] (v2) {};
				\node at (-43.3,25) [vertex] (v3) {};
				\node at (-43.3,-25) [vertex] (v4) {};
				\node at (0,-50) [vertex] (v5) {};
				\node at (43.3,-25) [vertex] (v6) {};
				\node at (0,0) [vertex] (x0) {}; 
				\node at (80,0) [vertex] (x1) {};
				\node at (160,0) [vertex] (x2) {};
				\node at (240,0) [vertex] (x3) {};
				\node at (320,0) [vertex] (x4) {};
				\node at (400,0) [vertex] (x5) {};  
				\node at (480,-75) [vertex] (u1) {}; 
				\node at (480,-25) [vertex] (u2) {};  
				\node at (480,25) [vertex] (u3) {};  
				\node at (480,75) [vertex] (u4) {};       
				\path
				(v1) edge (v2)
				(v2) edge (v3)		
				(v3) edge (v4)		
				(v4) edge (v5)
				(v5) edge (v6)
				(v6) edge (v1)	
				(x0) edge (v1)
				(x0) edge (v2)
				(x0) edge (v3)
				(x0) edge (v4)
				(x0) edge (v5)
				(x0) edge (v6)
				(x0) edge (x1)
				(x1) edge (x2)
				(x2) edge (x3)
				(x3) edge (x4)
				(x4) edge (x5)
				(x5) edge (u1)
				(x5) edge (u2)
				(x5) edge (u3)
				(x5) edge (u4)
				
				;
			\end{tikzpicture}
			\caption{$F(16,6,7)$}
			\label{fig:flagellum}
		\end{figure}

		Finally for $2 \leq D \leq n-a-1$ we define the \emph{flagellum graph} $F(n,a,D)$ as follows. Take a path $P$ of length $D-2$ with vertices $\{ x_0,x_1,\dots , x_{D-2} \} $, where $x_i \sim x_{i+1}$ for $0 \leq i \leq D-3$. Let $C_s$ be a cycle with length $s = n-D-a+3$ and vertex set $\{ u_0,u_1,\dots , u_{s-1} \}$, where $u_i \sim u_{i+1}$ for $0 \leq i \leq s-1$ and addition is carried out modulo $s$. Join $x_0$ to every vertex of $C_s$, so that $C_s \cup \{ x_0\} $ induces a wheel.  Finally append a set $Q = \{ v_1,v_2,\dots, v_{a-2}\} $ of $a-2$ pendant edges to $x_{D-2}$; an example is shown in Figure~\ref{fig:flagellum}. This graph has order $n$ and diameter $D$; we now show that $\mono(F(n,a,D)) = a$.  
		
		It is simple to verify that the set $Q \cup \{ u_0,u_1\} $ is an mp-set and so $\mono(F(n,a,D)) \geq a$. Suppose that $M$ is an mp-set with size $\geq a+1$. $M$ contains at most two points of $P$. If $M$ contains two points of $P$ then $M \cap (C_s\cup Q) = \emptyset $ and $|M| = 2 < a$.  Suppose that $M$ contains a point of $P$.  Then $M$ cannot contain points of both $C_s$ and $Q$ or else the point of $P$ contained in $M$ would lie on a monophonic path between any point of $M$ in $C_s$ and any point of $M$ in $Q$. Any mp-set contains at most two points of $C_s$, so if $M$ contains a point of $C_s$, it follows that $a+1 \leq |M| \leq 3 \leq a$, which is impossible. Therefore $M$ must contain a point of $Q$; but as $|Q| = a-2$  we have $|M| \leq 1+(a-2) = a-1$. Thus $M \cap P = \emptyset $. Therefore $M$ consists of at most two points of $C_s$ and at most $a-2$ points of $Q$, so that $|M| \leq 2+(a-2) = a$.    
	\end{proof}
	We now determine the possible diameters of graphs with mp-number two.
	
	\begin{theorem}
		There exists a graph with order $n$, monophonic position number $a = 2$ and diameter $D \geq 3$ if and only if $D = n-1$ or $3 \leq D \leq \lfloor \frac{n}{2} \rfloor $. 
	\end{theorem}
	\begin{proof}
		It follows from~\cite{ThomasChandranTuite} that for $n \geq 2r+1$ and $r \geq 3$ we have $\mono(H(n,r)) = 2$, where $H(n,r)$ is the half-wheel graph defined in Section~\ref{smallest graphs}. Moreover $H(n,r)$ has order $n$ and diameter $D = 3+\lceil \frac{n}{2} \rceil -r$ if $r \geq 4$ and diameter $\lceil \frac{n}{2}\rceil -1$ if $r = 3$. Varying $r$ from $3$ to $\lfloor \frac{n-1}{2} \rfloor $ we obtain graphs with order $n$, mp-number $2$ and all diameters in the range $4 \leq D \leq \lceil \frac{n}{2} \rceil -1$. For odd $n$ we have $\lceil \frac{n}{2} \rceil -1 = \lfloor \frac{n}{2} \rfloor $. For even $n \geq 4$ the cycle $C_n$ has mp-number $2$ and diameter $\lfloor \frac{n}{2} \rfloor $. 
		
		Theorem~\ref{turantheorem} shows that for $n \geq 3$  the graph $T^*_{n,2}$, which is a complete bipartite graph $K_{\lceil \frac{n}{2} \rceil , \lfloor \frac{n}{2} \rfloor }$ minus a matching of size $\lfloor \frac{n}{2} \rfloor $, has mp-number $2$; as this graph has diameter $D = 3$, we have proven existence for all $D$ in the range $3 \leq D \leq \lfloor \frac{n}{2} \rfloor $. For $n \geq 2$ the path with length $n-1$ is a graph with order $n$, mp-number $2$ and diameter $n-1$. This proves the existence of graphs with mp-number $2$ and all values of the diameter $D$ claimed in the statement of the theorem.
		
		We now show that there is no graph with order $n$, mp-number $2$ and diameter $D$, where $\lfloor \frac{n}{2} \rfloor < D < n-1$. Suppose that $G$ is such a graph and let $u$ and $v$ be vertices at distance $D$. Assume that $G$ is 2-connected. Then $u$ and $v$ are joined by internally disjoint paths of length $\geq D$, so that $n \geq 2D > n$, which is impossible. 
		
		Hence $G$ contains a cut-vertex $w$. Suppose that $d(w) \geq 3$. Then choose a set $M$ of three neighbours of $w$ such that the vertices of $M$ are not all contained in the same component of $G-w$; it is easily seen that $M$ is an mp-set. Hence we must have $d(w) = 2$. Each neighbour of $w$ is either a leaf of $G$ or a cut-vertex, so repeating this reasoning shows that $G$ is a path, which contradicts $D < n-1$.    
	\end{proof}

	\begin{lemma}\label{diameter two construction}
		If there exists a graph with order $r \geq 4$, monophonic position number $a = 2$ and diameter $D = 2$, then there exist graphs with monophonic position number $a = 2$, diameter $D = 2$ and orders $3r$, $3r+1$ and $3r+2$.
	\end{lemma}	
	\begin{proof}
		Let $H$ be a graph with order $r \geq 4$, monophonic position number $2$ and diameter $D = 2$. Label the vertices of $H$ as $h_1,h_2,\dots ,h_r$. We will construct new graphs with orders $3r$, $3r+1$ and $3r+2$ with mp-number $2$ and diameter $D = 2$ from $H$ as follows. 
		
		First we define the graph $G(H)$ with order $3r+2$. Let $X = \{ x_1,x_2,\dots ,x_r\} $ and $Y = \{ y_1,y_2,\dots,y_r\} $ be two new sets of vertices disjoint from $V(H)$. On $X \cup Y$ draw a complete bipartite graph with partite sets $X$ and $Y$ and then delete the perfect matching $x_iy_i$, $1 \leq i \leq r$. For $1 \leq i \leq r$ join both $x_i$ and $y_i$ to the vertex $h_i$ by an edge. Finally add two new vertices $z_1$ and $z_2$, join $z_1$ to each vertex of $X$ by an edge, join $z_2$ to each vertex of $Y$ and lastly add the edge $z_1z_2$ between the two new vertices. An example of this construction for $H = C_4$ is displayed in Figure~\ref{fig:3r+2 construction}.
		
		It is easily seen that $G$ has diameter $D = 2$. To show that the mp-number of $G$ is $2$ it is sufficient to show that for any set $M$ of three vertices of $G$ there is an induced path containing each vertex of $M$.  It is evident that for any vertex $v \in V(G) - \{ z_1,z_2\} $ there is an induced path in $G$ containing $z_1,z_2$ and $v$, so we can assume that $|M \cap \{ z_1,z_2\} | \leq 1$.

		The map fixing every element of $H$, interchanging $x_i$ and $y_i$ for $1 \leq i \leq r$ and swapping $z_1$ and $z_2$ is an automorphism of $G$, which reduces the number of cases that we need to check. Suppose that $M$ is a set of three vertices of $G(H)$ containing one of $z_1,z_2$; say $z_1 \in M$. Without loss of generality we have the following nine possibilities for $M' = M-\{ z_1\} $: i) $M' = \{ x_1,x_2 \} $, ii) $M' = \{ x_1,h_1\} $, iii) $M' = \{  x_1,h_2\} $, iv) $M' = \{ x_1,y_1\} $, v) $M' = \{ x_1,y_2\} $, vi) $M' = \{ h_1,h_2\} $, vii) $M' = \{ h_1,y_1\} $, viii) $M' = \{ h_1,y_2\} $ and ix) $M' = \{ y_1,y_2\} $. 
		
		Consider the following two cycles. For $1 \leq i,j \leq r$, where $i \not = j$, we define $C(i,j)$ to be the cycle $z_1,x_i,y_j,h_j,x_j,z_1$ and, if $P$ is a shortest path in $H$ from $h_i$ to $h_j$, then $D(i,j)$ is the cycle formed from the path $P$ from $h_i$ to $h_j$, followed by the path $h_j,x_j,z_1,x_i,h_i$. Both of these cycles are induced and so can contain at most two points of $M$. By varying the parameters $i$ and $j$ we see that the first seven configurations for $M'$ above are not possible. 
		
		For viii) let $P$ be a shortest path in $H$ from $h_1$ to $h_2$. By assumption $r \geq 4$, so as $P$ has length at most two, there exists a vertex of $H$, say $h_3$, not appearing in $P$. Then the path $P$, followed by the path $h_2,y_2,x_3,z_1$ contains all three vertices of $M$.  Finally for case ix) the induced path $y_1,x_2,z_1,x_1,y_2$ contains all three vertices of $\{ z_1,y_1,y_2\} $. 
		
		We can now suppose that $M \cap \{ z_1,z_2\} = \emptyset $. Observe that the subgraph of $G$ induced by $X \cup Y$ is isomorphic to the graph $T^*_{2r,2}$ from the proof of Theorem~\ref{turantheorem}, so we can also assume that $M \not \subseteq X \cup Y$. Furthermore, as $\mono(H) = 2$, we can take $M \not \subseteq V(H)$. For all $1\leq i,j \leq r$ and $i \not = j$ there is an induced cycle $x_i,h_i,y_i,x_j,h_j,y_j,x_i$, so $M$ must contain vertices with at least three different subscripts $i \in \{ 1,\dots ,r\} $. Therefore without loss of generality we are left with the following three cases: i) $M = \{ x_1,x_2,h_3\} $, ii) $M = \{ x_1,y_2,h_3\} $ and iii) $M = \{ x_1,h_2,h_3\} $.
		
		For cases i) and ii), let $P'$ be the shortest path in $H$ from $h_3$ to $\{ h_1,h_2\} $; without loss of generality $P'$ is a $h_2,h_3$-path that does not pass through $h_1$. Then in case i) $x_1,z_1,x_2,h_2$ followed by $P'$ contains all three points of $M = \{ x_1,x_2,h_3\} $ and in case ii) the path $x_1,y_2,h_2$ followed by $P'$ contains all three vertices of $M = \{ x_1,h_3,y_2\} $. 
		
		For case iii), if $P'$ is a shortest $h_2,h_3$-path in $H$, then $x_1,y_2,h_2$ followed by $P'$ contains all three vertices of $M = \{ x_1,h_2,h_3 \} $ unless $d(h_2,h_3) = 2$ and $P'$ is the path $h_2,h_1,h_3$, in which case the path $h_3,y_3,x_1,y_2,h_2$ suffices.   
		
		This analysis also shows that the graphs $G'(H) = G(H)-\{ z_2\} $ with order $3r+1$ and the graph $G''(H) = G(H)-\{ z_1,z_2\} $ with order $3r$ also have mp-number $2$ and  diameter $D = 2$. Hence for any $r \geq 4$ if there is a graph with order $r$, mp-number $2$ and diameter $D = 2$ there also exists such a graph for orders $3r, 3r+1$ and $3r+2$. \end{proof}

	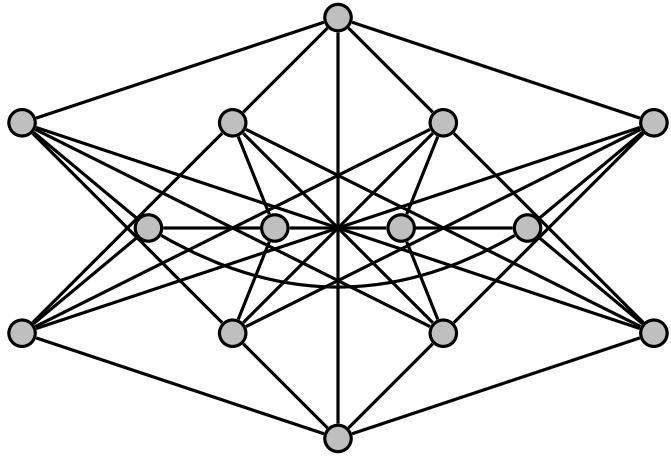
\begin{figure}
		\centering
		\begin{tikzpicture}[x=0.4mm,y=-0.4mm,inner sep=0.2mm,scale=0.7,very thick,vertex/.style={circle,draw,minimum size=10,fill=lightgray}]
			\node at (-150,-50) [vertex] (x1) {$$};
			\node at (-50,-50) [vertex] (x2) {$$};
			\node at (50,-50) [vertex] (x3) {$$};
			\node at (150,-50) [vertex] (x4) {$$};

			\node at (-90,0) [vertex] (h1) {$$};
			\node at (-30,0) [vertex] (h2) {$$};
			\node at (30,0) [vertex] (h3) {$$};
			\node at (90,0) [vertex] (h4) {$$};

			\node at (-150,50) [vertex] (y1) {$$};
			\node at (-50,50) [vertex] (y2) {$$};
			\node at (50,50) [vertex] (y3) {$$};
			\node at (150,50) [vertex] (y4) {$$};
			
			\node at (0,-100) [vertex] (z1) {$$};
			\node at (0,100) [vertex] (z2) {$$};
			
			\path
			(x1) edge (y2)
			(x1) edge (y3)
			(x1) edge (y4)
			
			(x2) edge (y1)
			(x2) edge (y3)
			(x2) edge (y4)
			
			(x3) edge (y1)
			(x3) edge (y2)
			(x3) edge (y4)
			
			(x4) edge (y1)
			(x4) edge (y2)
			(x4) edge (y3)
			
			(x1) edge (h1)
			(y1) edge (h1)
			(x2) edge (h2)
			(y2) edge (h2)
			(x3) edge (h3)
			(y3) edge (h3)
			(x4) edge (h4)
			(y4) edge (h4)
			
			(h1) edge (h2)
			(h2) edge (h3)
			(h3) edge (h4)
			(h4) edge [bend left] (h1)
			
			(z1) edge (x1)
			(z1) edge (x2)
			(z1) edge (x3)
			(z1) edge (x4)
			
			(z2) edge (y1)
			(z2) edge (y2)
			(z2) edge (y3)
			(z2) edge (y4)
			
			(z1) edge (z2)
			
			;
		\end{tikzpicture}
		\caption{The graph $G(C_4)$}
		\label{fig:3r+2 construction}
	\end{figure}

	\begin{theorem}
		There is a graph with order $n$, monophonic position number $a = 2$ and diameter $D = 2$ if and only if $n \in \{ 3,4,5,8\} $ or $n \geq 11$.
	\end{theorem}
	\begin{proof}
		The statement of the theorem has been verified by computer search for all $n \leq 32$~\cite{Erskine}. Let $n \geq 33$ and assume that the result is true for all orders $< n$. Write $n = 3r+s$, where $s$ is the remainder on division of $n$ by $3$. Then $r \geq 11$ and by the induction hypothesis there exists a graph with order $r$, mp-number $2$ and diameter $D = 2$. Then by Lemma~\ref{diameter two construction} there exists a graph with order $n$, mp-number $2$ and diameter $D = 2$. The theorem follows by induction.
	\end{proof}
	
	The computer search used in this proof suggests the following stronger result.
	\begin{conjecture}
		For any $n \geq 11$, there is a circulant graph with order $n$, monophonic position number $a = 2$ and diameter $D = 2$.
	\end{conjecture}

	{\bf Acknowledgements}
	The authors thank Dr. Erskine for his help with the computational results used in this paper. The first author acknowledges funding by an LMS Early Career Fellowship (Project ECF-2021-27) and also thanks the Open University for an extension of funding.  The second author thanks the University of Kerala for providing JRF.

	%%%%%%%%%%%%%%%%%%%%%%%%%%%%%%%%%%%%%%%%%%%%%%%%%%%%%%%%%%%%%%%%%%%%%%%%
	%%%%%%%%%%%%%%%%%%%%%%%%%%%%%%%%%%%%%%%%%%%%%%%%%%%%%%%%%%%%%%%%%%%%%%%%

\end{document}